\documentclass[english,reqno]{amsart}

\usepackage{latexsym}
\usepackage{amssymb,amsbsy,amsmath,amsfonts,amssymb,amscd}
\usepackage{mathrsfs}
\usepackage{epsfig, graphicx}
\usepackage{graphics,epsfig}
\usepackage{color}

\usepackage{subfigure}

\usepackage{amsthm}
\usepackage{amsfonts}
\usepackage{graphicx}
\usepackage{subfigure}
\usepackage{amsmath}
\usepackage{amssymb}
\usepackage{amsmath}
\usepackage{bm}
\usepackage{subeqnarray}
\usepackage{algorithm2e}
\usepackage{algpseudocode}
\usepackage{cases}
\theoremstyle{plain}

\newtheorem{theorem}{Theorem}
\newtheorem{assumption}{Assumption}
\newtheorem{lemma}{Lemma}

\newtheorem{corollary}{Corollary}

\theoremstyle{definition} 

\newtheorem{remark}{Remark}

\renewcommand{\vec}[1]{\mathbf{#1}}
\newcommand{\eps}{\varepsilon}
\newcommand{\average}[1]{ \langle#1 \rangle}
\newcommand{\norm}[1]{ \left\| #1 \right\|}

\newcommand{\half}{\frac{1}{2}}

\newcommand{\NullL}{{\rm Null} \,\Lop}
\newcommand{\NullT}{{\rm Null} \,\Top}
\newcommand{\Span}{{\rm Span}}

\newcommand{\Top}{\mathsf{T}}
\newcommand{\Lop}{\mathsf{L}}
\newcommand{\Aop}{\mathsf{A}}
\newcommand{\Kop}{\mathsf{K}}
\newcommand{\II}{\mathsf{I}}

\newcommand{\Hfun}{\mathcal{H}}
\newcommand{\Dfun}{\mathcal{D}}
\newcommand{\Source}{\mathcal{S}}
\newcommand{\Cto}{\tilde{C}_1}
\newcommand{\Ctt}{\tilde{C}_2}
\newcommand{\tildeh}{\tilde{h}}
\newcommand{\tildeg}{\tilde{g}}

\newcommand{\rd}{\mathrm{d}}

\newcommand{\RR}{\mathbb{R}}
\newcommand{\Kn}{\mathsf{Kn}}

\graphicspath{{../}}

\title{Uniform regularity for linear kinetic equations with random input based on hypocoercivity}
\author{Qin Li} 
\address{Mathematics Department, University of Wisconsin-Madison, 480 Lincoln Dr., Madison, WI 53705 USA.}
\email{qinli@math.wisc.edu}
\author{Li Wang} 
\address{Department of Mathematics, Computational and Data-Enabled Science and Engineering Program, State University of New York at Buffalo, 244 Mathematics Building, Buffalo, NY 14206 USA.}
\email{lwang46@buffalo.edu}
\thanks{The work of Q. L. is supported in part by a start-up fund from UW-Madison and National Science Foundation under the grant DMS-1619778. The work of L.W. is supported in part by the National Science Foundation under the grant DMS-1620135. Both Q.L. and L.W. are grateful to Prof. Shi Jin's long term support and inspiring discussions.}

\begin{document}
\maketitle

\begin{abstract}
In this paper we study the effect of randomness in kinetic equations that preserve mass. Our focus is in proving the analyticity of the solution with respect to the randomness, which naturally leads to the convergence of numerical methods. The analysis is carried out in a general setting, with the regularity result not depending on the specific form of the collision term, the probability distribution of the random variables, or the regime the system is in, and thereby termed ``uniform". Applications include the linear Boltzmann equation, BGK model, Carlemann model, among many others; and the results hold true in kinetic, parabolic and high field regimes. The proof relies on the explicit expression of the high order derivatives of the solution in the random space, and the convergence in time is mainly based on hypocoercivity, which, despite the popularity in PDE analysis of kinetic theory, has rarely been used for numerical algorithms.
\end{abstract}

\section{Introduction}
Kinetic equation is a set of equations that describe the collective behavior of many-particle systems. The solution to the equation is typically defined on the phase space, characterizing the evolution of the probability distribution. Depending on the particle system one is looking at, scientists derived radiative transfer equation for photons, the Boltzmann equation for rarified gas, the Fokker-Planck equation for plasma, run-and-tumble models for bacteria and many others.

Uncertainty is a nature of kinetic theory. The modeling error, the blurred measurements of coefficients in the equation, and the empirical constitutive relations all contribute to inaccuracy in the solution. Yet it is not realistic to look for the exact true solution, we instead are more concerned on quantifying the uncertainties and approximately obtaining the solution behavior in the probability sense.

Many numerical techniques have been developed to address the issues related to the uncertainties, among which, we specifically mention generalized polynomial chaos method (gPC)~\cite{GHANEM1996289,ghanem_construction_2006,Xiu_gpc, XK02}, stochastic collocation method~\cite{babuska_stochastic_2007,XiuHesthaven_collocation}, and Monte Carlo method with its many variations~\cite{fishman2013monte,Giles_MLMC,Barth2011,Charrier_MLMC}. The latter two are categorized as non-intrusive, meaning that the implementation of the algorithm simply calls for deterministic solver repeatedly, while the first one is intrusive, wherein a completely new implementation is needed. Monte Carlo method is a traditional method for handling uncertainties, but with a major drawback of slow convergence rate. On the other hand, both polynomial chaos method and stochastic collocation method are some variations of the spectral or psudo-spectral method applied along the random dimension, and automatically inherit the fast convergence. However, the assertions on the efficiency do heavily rely on the assumption that the solution has certain regularity along the random space, which needs to be justified case by case.

In the past few years, we have seen many such verifications for several different types of equations, including~\cite{BNT07,BNT04, ZG12, cohen_analytic_2011, cohen_convergence_2010}, and these analysis sometimes suggest new algorithms that better explore the solution structure~\cite{hou_li_zhang16_1,hou_li_zhang16_2,chkifa_breaking_2014,NTW08,nobile_sparse_2008,nobile_anisotropic_2008,SchwabTodor_sparse, GPZ15,DP_uncertaity}. The developments seems to have been concentrated on the elliptic type or parabolic type of equations. For a long time the similar treatment to hyperbolic type of equation has been left blank due to its intrinsic difficulty~\cite{majda_gpc,DP_uncertaity}: the solution develops non-smooth structure, breaking the assumptions the spectral methods rely on.

We study the regularity on the random space for the kinetic equation in this paper. Besides the fact that kinetic equation naturally contains many aspects of uncertainties, and is an interesting topic on its own, the study also serves as a building block in understanding the randomness' influence in the passage from hyperbolic to parabolic types. Indeed, there are certain parameters in the kinetic equations, adjusting which one moves the equation across regimes. One typical example is the high field regime of kinetic equations wherein time and space are rescaled in the same fashion, and the limiting equation falls in the hyperbolic category, whereas keeping the space scale and elongating the time scale, the equation moves into the diffusive regime and approximates a diffusion type equation. We investigate in this paper the response of the solutions' regularity to the parameters used to perform rescaling, and we study if it is possible to build a general framework that can be applied without the dependence of the regimes the equation is in. Some recent results on the topic can be found in~\cite{JXZ15, HJ16,JLu16,JZ16, JLM16,JL16}, however, the proofs are accomplished on a case-by-case basis, and not necessarily in their sharpest estimates, especially in the big space long time regime. Among them, two papers are of special interest. In~\cite{JLM16}, the authors first successfully controlled the regularity in long time, and provided a uniform convergence of stochastic Galerkin method applied on the radiative transfer equation in the presence of both kinetic and diffusive regimes. It was followed by~\cite{JL16} in which the authors gave a bounded estimate of the solution under a specially chosen weighted norm for the semiconductor Boltzmann equation. In this paper, we intend to provide a general framework and sharpest estimates in this general setting for all scales. \emph{More specifically, we aim at conducting analysis for the kinetic equation in its abstract form, and study the regularity of the solution in the random space in all regimes.} The main results are summarized as follows 
\begin{theorem}
(Informal version) Let $f$ be the solution to the kinetic equation \eqref{eqn:f}, and assume the initial data has sufficient regularity with respect to the random variable $z$, i.e., $\|\partial_z^l f_0\| \leq H^l$,
then:
\begin{itemize}
\item[(1)] the $l-$th derivative in $z$ of $f$ has the estimate: $$\left\| \partial_z^l f \right\| \leq C l! \min \{ e^{-\lambda_z t} C(t)^l, \ e^{(C-\lambda_z)t} 2^{l-1} (1+H)^{l+1} \}\,,$$ where $C$ is a constant, $C(t)$ is an algebraic function of $t$, and $\lambda_z > 0$ is uniformly bounded below away from zero;
\item[(2)] $f$ is analytic with uniform convergence radius $\frac{1}{2(1+H)}$;
\item[(3)] both the exponential convergence in time and convergence radius are uniform with respect to the Knudsen number that characterize different macroscopic regimes including diffusive regime and high field regime. 
\end{itemize}
\end{theorem}
Here in (1) the former bound indicates the exponential decay in time, whereas the latter one guarantees the analyticity of $f$ stated in (2). The third part (3) then expresses the above results are valid across different macroscopic regimes. This is the main emphasis of this paper. More detailed explanations can be found in Theorems \ref{thm:gl1}---\ref{thm:high-322}.

The idea behind our proof is two-folds: the use of the hypocoercivity guarantees the decay in time that gets rid of the small parameter dependence, and the careful hierachical derivation provides the explicit dependence on the randomness that also leads to sharp estimates. The use of the hypocoercivity is done with caution the brute-force computation by Gronwall inequality gives $e^{\frac{t}{\eps}}$ growth for all derivatives, which as shown in this paper is far from optimal. In~\cite{DMS15} the authors defines a modified $L_2$ norm that allows us to find the explicit decay rate, and delicate derivation in this paper shows its independence on the rescaling parameters, allowing us to pass regimes.

We lay out the equation and its basic assumptions in Section 2, together with detailed studies of the convergence rate in time in the deterministic setting. Section 3, 4 and 5 are respectively devoted to the study extended to equations in various of regimes,  to equations involving randomness, and to scenarios when both present. We conclude in Section 6.

\section{Basic assumptions and contractivity}
We specify notations, basic assumptions, and briefly recall the properties of solutions in this section. Consider linear kinetic equations with random input in full generality:
\begin{equation}\label{eqn:f}
\partial_t f + \Top f = \Lop_z f\,,\quad f = f(t,x,v,z)\geq 0\,,\quad (t,x,v,z)\in\RR^+\times\RR^d\times\RR^d\times\Omega\,,
\end{equation}
with initial data
\begin{equation*}
f(t=0,x,v,z) = f_0(x,v,z)\in L^2(\rd{x}\rd{v})\,.
\end{equation*}
It describes the evolution of a distribution function $f$ at time $t$ on phase space $(x,v)$, and subject to a set of random variables $z$. The operators $\Top$ and $\Lop_z$ are typically called the transport term and the collision operator, which represent the streaming of particles along the Hamiltonian flow, and the interactions with the background:
\begin{itemize}
\item $\Top$ represents the transport term. For a free transport it is simply
\begin{equation*}
\Top f = v\cdot\nabla_xf\,,
\end{equation*}
whereas with external potential $V(x)$, it describes a flow driven by the Hamiltonian:
\begin{equation*}
H (x,v) = \frac{1}{2}|v|^2+V(x)\,,
\end{equation*}
and writes as:
\begin{equation}  \label{eqn:Top}
\Top = v\cdot\nabla_x -\nabla_xV\cdot\nabla_v\,.
\end{equation}
Due to the nature of the transport term, it is always skew symmetric.
\item $\Lop_z$ is a collision operator, and we assume it acting purely on $v$. The subindex $z$ of $\Lop_z$ stands for the random dependence. 
Depending on specific applications, it has varies forms, including:
\begin{itemize}
\item BGK type operator:
\begin{equation*}
\Lop_z f= \sigma(x,z)(\Pi f - f)\,.
\end{equation*}
where $\Pi$ is a projection operator and $\sigma$ -- the so-called scattering coefficient -- may have both spatial and random dependence. The specific form of $\Pi$ varies according to the particle system the equation describes;
\item Anisotropic scattering operator: 
\begin{equation} \label{eqn:1.7}
\Lop_z f(v)= \int[k_z(v^\ast\to v)f(v^\ast) - k_z(v\to v^\ast)f(v)]\rd{v}^\ast\,, \quad k_z >0\,.
\end{equation}
We put $z$ as a subscript of $k$ to indicate the random dependence. It is slightly more general than the BGK operator.
\item Fokker-Planck operator:
\begin{equation*}
\Lop_z = \sigma(x,z)\left[\nabla_v\cdot(\nabla_v f+vf)\right]\,.
\end{equation*}
\end{itemize}
The collision takes place at the microscopic level, and the operator is always symmetric. In this paper we only consider the first two cases, both of which only preserve mass, and thus the collision operator has one dimensional null space and integrate to zero, which will be specified later. In what follows, $\Lop$ will also be used later to represent the classical deterministic collision without $z$ dependence. 
\end{itemize}

To be more specific on the random dependence, we let $z\in\Omega$ be a set of finitely many random variables. Unlike the previous papers \cite{JLM16, JLW16}, here we do not specify a probability measure on $\Omega$, so that our theory developed in this paper can be applied to arbitrary probability space. Nevertheless, we do require it appearing only in the collision operator or the initial datum, but not in the transport operator, which we leave to future study.

As mentioned above, it is the multiple scales indicated by the magnitude of a dimensionless parameter, that makes the problem interesting and challenging. In the equations we considered here, we denote such parameter the Knudsen number $\Kn$, which represents the ratio of mean free path and the typical domain length. Upon non-dimensionalization, equation~\eqref{eqn:f} reads
\begin{equation} \label{eqn:000}
\partial_t f + \frac{1}{\Kn}\Top f = \frac{1}{\Kn^2}\Lop_z f\,
\end{equation}
in the parabolic scaling, and
\begin{equation} \label{eqn:001}
\partial_t f + \frac{1}{\Kn} \Top f = \frac{1}{\Kn}\Lop_z f\,
\end{equation}
in the ``high field" scaling. The deterministic version of the former has been intensively studied in the literature \cite{BSS84}, and the latter was first considered by Poupaud \cite{Poupaud92} for the Fokker-Planck case, then by Cercignani et. al. \cite{CGL97} for semiconductor Boltzmann equation, and more recently investigated in \cite{JW11, JW13, JZ16}. Note that in the classical high field regime only the field term in the transport operator (i.e. $-\nabla _x V \cdot \nabla_v$ in (\ref{eqn:Top}) ) is rescaled by $\Kn$, here we use $\frac{1}{\Kn}\Top f$ to rescale both terms just to lighten the notation. In fact, the analysis in the following sections that consider this scaling can be easily adapted to the classical high field scaling.

\subsection{Equation properties}
To quantify the uncertainties' propagation along the random space, several questions need to be addressed: given that initial data and $\Lop_z$ having smooth dependence on $z$, does the solution $f$ remain smooth in $z$? How does the regularity change according to $\Kn$? Before setting out to understand these questions, we first restrict our attention to the deterministic version of~\eqref{eqn:f}, for which we will explicitly find the solution's decay rate in time. The analysis largely relies on~\cite{DMS15}, but in an exposition to facilitate our analysis later for cases including multiple scales and randomness. 

We first define {\it local equilibrium}. It is a collection of functions that diminish the effect of the collision term, and we denote $\NullL$ the null space of $\Lop$:
\begin{equation*}
\NullL = \{f\in L^2(\rd{x}\rd{v}): \Lop f = 0\}\,.
\end{equation*}
As the equation only preserves the mass, the null space can be simply constructed as:
\begin{equation*}
\NullL = \Span\{\mathcal{M}(x,v)\} = \{\rho(t,x)\mathcal{M}(x,v)\}\,,
\end{equation*}
where $\int \!\! \int \mathcal{M}(x,v)\rd{x}\rd{v} = 1$ is normalized. Note that $\Lop$ is an operator that acts only on $v$, which allows us to separate out the $\mathcal{M}(x,v)$ term. The associated projection operator is immediate:
\begin{equation*}
\Pi f (t,x,v)= \frac{\int f(t,x,v)\rd{v}}{\int \mathcal{M}(x,v)\rd{v}}\mathcal{M} (x,v)\,.
\end{equation*}

Considering the equation conserving mass, meaning
\begin{equation*}
\int_{\RR^d}\Lop f\rd{v} = 0\,,
\end{equation*}
the total mass remains a constant, namely:
\begin{equation*}
\frac{\rd}{\rd t} M= \iint_{\RR^d\times\RR^d}(\Lop-\Top)f\rd{v}\rd{x} = 0\,,
\end{equation*}
where
\begin{equation*}
M = \frac{\rd}{\rd t}\iint_{\RR^d\times\RR^d}f\rd{v}\rd{x}
\end{equation*}
is the total mass.

We then define the {\it global equilibrium}. It is a collection of functions that live in the intersection of the two null spaces:
\begin{equation*}
\NullL\cap\NullT = \Span\{F\}\,.
\end{equation*}
We require $F$ strictly positive, integrable, and normalized: 
\begin{equation*}
\iint_{\RR^d\times\RR^d} F(x,v)\rd{v}\rd{x} = 1\,.
\end{equation*}

With the dissipative assumption that is satisfied by many collision operators, a vast of literature have addressed the convergence of $f$ towards the global equilibrium. That is, given arbitrary $f_0\in L^2$, the solution of (\ref{eqn:f}) converges to the global Maxwellian:
\begin{equation*}
f(t,x,v)\to MF(x,v)\,,\quad t\to \infty\,.
\end{equation*}
Such examples include \cite{Ukai74, CCG03} for linearized or linear  Boltzmann equation, \cite{DV01, HN04} for Fokker-Planck equation, \cite{DV05} for spatially-inhomogeneous Boltzmann equation, and etc. Among them, we would like to point out \cite{DMS15}, in which the authors provide a decay rate via a unified framework that works for a large class of linear kinetic equations. Our theory will be constructed based on this work. 

Since the equation~\eqref{eqn:f} is linear, the fluctuations around the equilibrium (i.e. $f-MF$) follows the same equation. And for easier notation, we will consider the {\it fluctuations} rather than the function itself. With a little abuse of notation, we still denote the fluctuation as $f$, then it has zero mass 
\begin{equation}
M = \int\int_{\mathbb R^d \times R^d} f(t,x,v) \rd{x}\rd{v} = 0\,.
\end{equation}
Due to the convergence towards the global Maxwellian $F$, it is natural to change the Lebesgue measure to the following:
\begin{equation}
\rd\mu = \rd\mu(x,v) = \frac{\rd{x}\rd{v}}{F}\,,\quad (x,v)\in\RR^d\times\RR^d
\end{equation}
and the Hilbert space $\mathcal{H}= L^2(F^{-1}\rd{x}\rd{v})$ is endowed with the norm $\|\cdot\|$ with respect to the following inner product:
\begin{equation} \label{eqn:322-2}
\langle f\,,g\rangle = \iint_{\RR^d\times\RR^d} fg\rd\mu\,.
\end{equation}

\subsection{Assumptions}
We now list all assumptions for the kinetic equation we study. They are formulated in the abstract form, which need to be justified for different models individually. As already checked in~\cite{DMS15},  almost all the kinetic equations we have encountered satisfy these assumptions.
 
\begin{assumption}[Microscopic coercivity]\label{ass:1} The operator $\Lop$ is symmetric and there exists $\alpha>0$ such that
\begin{equation}\label{eqn:mic_coercive}
-\langle\Lop f\,,f\rangle\geq\alpha\|(\II-\Pi)f\|^2\,, \quad \text{ for all} \quad f \in D(\Lop)\,
\end{equation}
where $D(\Lop)$ represents the domain of $\Lop$. 
This assumption basically requires a spectral gap on $\NullL^\perp$. For simplicity of notation later we just denote $\alpha$ the biggest possible such constant.
\end{assumption}
\begin{assumption}[Macroscopic coercivity]\label{ass:2} The operator $\Top$ is skew symmetric and there exists $\beta>0$ such that
\begin{equation}\label{eqn:mac_coercive}
\|\Top\Pi f\|^2\geq\beta \|\Pi f\|^2\,, \quad \text{for all} \quad f \in \mathcal{H} \quad \text{s.t. } \quad \Pi f \in D(\Top)\,.
\end{equation}
Since $\Pi f$ typically provides the local equilibria that is equivalent to macroscopic quantities, and $\Top$ is a transporting operator, this assumption is very similar to the Poincar\'e inequality, which on the rough level, states that the derivatives are ``larger" than the quantity itself. Similar as above, for the simplicity of notation later we denote $\beta$ the biggest possible such constant.
\end{assumption}
\begin{assumption}[Orthogonality]\label{ass:3}
\begin{equation}
\Pi\Top\Pi = 0\,.
\end{equation}
This assumption indicates all functions, when projected in $\NullL$, and move along the flow, will be perpendicular to $\NullL$.
\end{assumption}
Denote 
\begin{equation}\label{eqn:Aop}
\Aop = \left(1+(\Top\Pi)^\ast(\Top\Pi)\right)^{-1}(\Top\Pi)^\ast\,,
\end{equation}
then we make
\begin{assumption}[Boundedness of auxiliary operator]\label{ass:4} The operator $\Aop\Top(1-\Pi)$ and $\Aop\Lop$ are both bounded, meaning that there exists $\gamma$ such that
\begin{equation}\label{eqn:boundedness}
\|\Aop\Top(1-\Pi)f\|+\|\Aop\Lop f\| \leq  \gamma\|(1-\Pi) f\|^2\,.
\end{equation}
The constructive definition of $\Aop$ is useful only in proving the following theorem. $\gamma$ also denotes the biggest possible such constant.
\end{assumption}

We directly cite the results from~\cite{DMS15} regarding the exponential decay of the fluctuation. 
\begin{theorem}\label{thm:hypocoercive}
Under the four assumptions, there exists $\lambda(\eps)$ and $C(\eps)$ that are explicitly computable in terms of $\alpha$, $\beta$, $\gamma$ and $\varepsilon$ such that for any initial datum $f(0,x,v)\in\mathcal{H}$,
\begin{equation}
\|f\|=\|e^{t(\Lop-\Top)}f_0\|\leq C(\eps)e^{-\lambda(\eps) t}\|f_0\|\,,
\end{equation}
where
\begin{equation}\label{eqn:c-const}
C(\eps) = \sqrt{\frac{1+\varepsilon}{1-\varepsilon}}\,,
\end{equation}
and $\eps\in[0,1)$ is chosen such that $\lambda(\eps)>0$.
\end{theorem}

The proof first appeared in~\cite{DMS15}. For completeness we still include the details, and we provide an explicit form of $\lambda$. As mentioned in \cite{DMS15}, the exponential decay rate may not be optimal, but it suffices our purpose.

\begin{proof}
Inspired by~\cite{Herau06}, the authors in \cite{DMS15} constructed the entropy function
\begin{equation} \label{eqn:entropy}
\Hfun (f) = \half \| f\|^2 + \eps \average{\Aop f, f}\,,
\end{equation}
where $\Aop$ is defined in (\ref{eqn:Aop}). Then we have
\begin{equation}
\frac{\rd}{\rd t}\Hfun[f] = -\Dfun[f]\,,
\end{equation}
where
\begin{equation} \label{eqn:D}
\Dfun [f] = -\langle \Lop f\,,f\rangle+\varepsilon\langle\Aop\Top\Pi f\,,f\rangle+\varepsilon\langle\Aop\Top(1-\Pi) f\,,f\rangle 
  -\varepsilon\langle\Aop\Lop f\,,f\rangle
- \varepsilon\langle\Top\Aop f\,,f\rangle        \,.
\end{equation}
With Assumption~\ref{ass:3}, one can show
\begin{equation} \label{eqn:000}
\|\Aop f\|\leq\frac{1}{2}\|(1-\Pi)f\|, \quad
 \|  \Top \Aop f  \| \leq \| (I-\Pi) f \|\,.
\end{equation}
Also, Assumption~\ref{ass:2} implies
\begin{equation}\label{eqn:001}
\average{\Aop \Top \Pi f, f} \geq \frac{\beta}{1+\beta} \|\Pi f\|^2\,.
\end{equation}
Collapsing the estimates in (\ref{eqn:000}) (\ref{eqn:001}) and Assumptions 1--4 into one equation, one gets
\begin{equation*}
\Dfun[f]\geq \left[\alpha -\varepsilon(1+\gamma)\left(1+\frac{1}{2\delta}\right) \right]\|(1-\Pi)f\|^2+\varepsilon \left[\frac{\beta}{1+\beta}-(1+\gamma)\frac{\delta}{2} \right]\|\Pi f\|^2\,.
\end{equation*}
Note the relation
\begin{equation}\label{eqn:relation1}
\frac{1}{2}(1-\varepsilon)\|f\|^2\leq \Hfun [f]\leq\frac{1}{2}(1+\varepsilon)\|f\|^2\,,
\end{equation}
we have
\begin{equation}\label{ineqn:entropy}
\frac{\rd}{\rd t}\Hfun[f]\leq -\frac{2\kappa}{1+\varepsilon}\Hfun[f]\,,
\end{equation}
with
\begin{equation} \label{eqn:kappa}
\kappa(\eps)=\min\left\{\alpha  -\varepsilon(1+\gamma)\left(1+\frac{1}{2\delta}\right)\,,\varepsilon \left[\frac{\beta}{1+\beta}-(1+\gamma)\frac{\delta}{2} \right]\right\}\, >0.
\end{equation}
The inequality~\eqref{ineqn:entropy} concludes the proof with 
\begin{equation} \label{eqn:lambda}
\lambda(\eps) = \max_{\delta } \frac{\kappa(\eps)}{ 1 + \eps}  
= \max_{\delta } \min 
\left\{\frac{\alpha  -\varepsilon(1+\gamma)\left(1+\frac{1}{2\delta}\right)}{1+\eps}\,, \quad \frac{\varepsilon}{1+\eps} \left[\frac{\beta}{1+\beta}-(1+\gamma)\frac{\delta}{2} \right]\right\}.
\end{equation}
\end{proof}

\begin{remark}
Several remarks are in order.
\begin{itemize}
\item Since $e^{0(\Lop-\Top)}=\II$ and $e^{(t+s)(\Lop-\Top)}=e^{t(\Lop-\Top)}e^{s(\Lop-\Top)}$, the operator $\Lop-\Top$ defines a semi-group. It being contractive has been shown in many other papers~\cite{Des06, DS09}, but the result above gives a computable rate.
\item Without constructing the new entropy function it is easy to see:
\begin{equation*}
\langle\partial_tf = (\Lop - \Top)f\,,f\rangle\quad\Rightarrow\quad\frac{1}{2}\frac{\rd}{\rd t}\|f\|^2 = \langle\Lop f\,,f\rangle\leq 0\,.
\end{equation*}
meaning that the solution decays in $\|\cdot\|$ norm. Here we have used the fact that $\Top$ is skew symmetric and $\Lop$ is coercive, which provides:
\begin{equation*}
\langle\Top f\,,f\rangle = 0\,, \quad\langle\Lop f\,,f\rangle \leq 0\,.
\end{equation*}
However, this analysis fails to characterize the decay in $\NullL$: we seek for a possible non-zero spectral gap type estimate to make the right hand side strictly negative. The new entropy~\eqref{eqn:entropy} provides this specific gap, at the cost of amplifying the norm by a constant $C$~\eqref{eqn:c-const}.
\item In the original paper \cite{DMS15} the authors simply stated that the rate is computable without providing a specific form. Its dependence on all possible parameters is not addressed either. In this paper, however, we need a more delicate estimate, and many details need to be filled in. More specifically,
\begin{itemize}
\item[(a)] \eqref{eqn:lambda} displays an intricate relation between $\lambda$ and $\eps$, as well as an implicit constraint on $\eps$ such that $\lambda(\eps)>0$. To get the fastest decay rate, we are expected to find
\begin{equation} \label{eqn:lambda-5}
\lambda = \max_{\eps}\lambda(\eps)=
\max_{\eps, \delta } \min 
\left\{\frac{\alpha  -\varepsilon(1+\gamma)\left(1+\frac{1}{2\delta}\right)}{1+\eps}\,, \quad \frac{\varepsilon}{1+\eps} \left[\frac{\beta}{1+\beta}-(1+\gamma)\frac{\delta}{2} \right]\right\}\,;
\end{equation}
\item[(b)] According to the definition of $C(\varepsilon)$ in~\eqref{eqn:c-const}, we need to make sure at the point $\lambda(\eps)$ achieves its maximum value, $\varepsilon$ needs to be strictly less than one.
\end{itemize}
These results will help us to get the uniform convergence with respect to $\Kn$ and $z$. On top of the explicit formulation found above, we also need to investigate how $\lambda$ and $\eps$ vary according to $\Kn$ and $z$. They are addressed in Section 3 and 4 respectively.

\item The framework gets easily adapted to torus case. We neglect such discussion in the current paper.
\end{itemize}
\end{remark}

\section{$\Kn$ dependence in deterministic setting}
In this section, we show that, in the absence of randomness, the contractive coefficient $\lambda$ that governs the exponential decay enjoys a uniform lower bound regardless of the magnitude of the Knudsen number $\Kn$. Hence we omit the subscript $z$ to indicate that there is no $z$ dependence here. Considering the explicit expression for $\lambda$ in~\eqref{eqn:lambda-5}, we only need to discuss 
\begin{itemize}
\item[(1)] how to solve the max min problem for the dependence of $\lambda(\eps)$ and $\eps$ on the coercive and boundedness parameters $\alpha$, $\beta$ and $\gamma$; 
\item[(2)] how these parameters change with respect to $\Kn$.
 \end{itemize}
 We answer these two questions in the following two subsections.

\subsection{Parabolic scaling}
In the parabolic scale, $\Top_\Kn\to  \frac{1}{\Kn}\Top$ and $\Lop_\Kn\to\frac{1}{\Kn^2}\Lop$, then according to the definition of $\alpha$, $\beta$ and $\gamma$ in~\eqref{eqn:mic_coercive},~\eqref{eqn:mac_coercive} and~\eqref{eqn:boundedness}, we have:
\begin{lemma}\label{lemma:rescale}
In the parabolic regime (\ref{eqn:000}), we have
\begin{equation*}
\alpha_\Kn =\frac{\alpha}{\Kn^2}\,,\quad \beta_\Kn =\frac{\beta}{\Kn^2}\,,\quad \gamma_\Kn  = \frac{\gamma}{\Kn}\,,
\end{equation*}
where $\alpha$, $\beta$ and $\gamma$ are the parameters when $\Kn = 1$.
\end{lemma}
\begin{proof}
Denote $\Lop_\Kn = \frac{1}{\Kn^2}\Lop$, then $\Lop_\Kn$ and $\Lop$ share the same null space, and for $f\in \NullL^\perp$ we have:
\begin{equation*}
\Lop_\Kn f =\frac{1}{\Kn^2}\Lop f\,,
\end{equation*}
and thus $\alpha_\Kn = \frac{1}{\Kn^2}\alpha$. Similarly denote $\Top_\Kn=\frac{1}{\Kn}\Top$ then the domain of $\Top_\Kn$ remains the same as that of $\Top$. For $\forall g \in \mathcal{H}$ such that $\Pi g \in D(\Top_\Kn)$, one has
\begin{equation*} 
\Top_\Kn \Pi g =\frac{1}{\Kn}\Top \Pi g\quad\Rightarrow\quad\|\Top_\Kn \Pi g\|^2 =\frac{1}{\Kn^2}\|\Top \Pi g\|^2\,,
\end{equation*}
indicating $\beta_\Kn = \frac{1}{\Kn^2}\beta$.

To understand $\gamma$, we first look at $\Aop$. Considering $\Top_\Kn = \frac{\Top}{\Kn}$, $\Aop_\Kn$ in the leading order as $\Kn\to 0$ becomes
\begin{equation}
\Aop_\Kn = \left(1+(\Top_\Kn\Pi)^\ast(\Top_\Kn\Pi)\right)^{-1}(\Top_\Kn\Pi)^\ast= \Kn\left(\Kn^2+(\Top\Pi)^\ast(\Top\Pi)\right)^{-1}(\Top\Pi)^\ast\sim\Kn\,.
\end{equation}
Putting it back to~\eqref{eqn:boundedness}: for $f\in\NullL$ both sides are zero, and for $f\perp\NullL$, $\Aop_\Kn\Top_\Kn$ gives roughly $\mathcal{O}(1)$ in $\Kn$ and $\Aop\Lop$ gives $\frac{1}{\Kn}$. These all combined lead to $\gamma_\Kn \sim\frac{1}{\Kn}\gamma$.
\end{proof}

We then solve the max min problem in~\eqref{eqn:kappa} for possible explicit expression of $\lambda$.
\begin{lemma}\label{lemma:lambda}
Denote 
\begin{equation} \label{eqn:abcd}
a = \alpha, \quad d = \frac{1+\gamma}{2}, \quad c = \frac{\beta}{1+\beta}\,. 
\end{equation}
Let 
\begin{equation}\label{eqn:k0}
k_0  = \max \left\{ 2, \quad \frac{2d^2}{ac} \right\}\,,
\end{equation}
then 
$\lambda$ to the max min problem \eqref{eqn:lambda-5} has a lower bound
\begin{equation}\label{eqn:lambda-0000}
\lambda \geq \frac{\tilde{a}d^2}{(k_0 a + \tilde{a})(k_0a +c)} \,, \qquad \tilde{a} = \frac{k_0 a^2 c}{k_0 ac +2dc}\,.
\end{equation}
\end{lemma}
\begin{proof}
Using the notations (\ref{eqn:abcd}), the max min problem becomes:
\begin{equation}
\lambda = \max_{\eps, \delta }\min  \frac{1}{1+\eps}\left\{  a-2b \left( 1+\frac{1}{2\delta} \right)\,, ~  \eps(c - d \delta) \right\}\,,
\end{equation}
where $b = \frac{\varepsilon(1+\gamma)}{2}$. Note that for a fixed $\eps$, $a -2b - \frac{\eps d}{\delta}$ is monotonically increasing in $\delta$ whereas $\eps (c - \delta b)$ is decresing. Thus the $\max_\delta \min$ takes place at their intersection. More specifically, we have
\begin{eqnarray} \label{eqn:1017}
\lambda (\eps) &=& \frac{1}{1+\eps} \max_\delta \min \left\{ a-2b \left( 1+\frac{1}{2\delta} \right)\,,   ~\eps(c - d \delta) \right\}  \nonumber
\\ &=&  \frac{1}{2}\left[ (a-2b +\eps c ) - \sqrt{(a-2b-\eps c )^2+4bd\eps} \right] \frac{1}{1+\eps} \,,
\end{eqnarray}
where the maximum in $\delta$ is taken at
\begin{equation*}
\delta =  \frac{-(a-2b - \varepsilon c)+\sqrt{(a-2b -\varepsilon c)^2+4bd\eps}}{2\eps d}\,.
\end{equation*}

Now it remains to find the maximum of \eqref{eqn:1017} in $\eps$, i.e, 
\begin{equation*}
\lambda = \max_\eps \lambda (\eps)\,.
\end{equation*}
Notice that if we take 
\begin{equation} \label{eqn:eps0-2}
\eps_0 = \frac{ac}{2dc + kd^2}\,,
\end{equation}
with $k$ an order one constant to be determined below, then it satisfies 
\begin{equation} \label{eqn:eps0-1}
\eps_0 = \frac{(a-2b)c}{kd^2} \,, \quad k \sim \mathcal{O}(1)\,.
\end{equation}
Plugging the above equation into (\ref{eqn:1017}), and denoting $\tilde{a} = a-2b$ evaluated at $\eps = \eps_0$, we get
\begin{eqnarray}
\lambda (\eps_0) &=& \frac{1}{2\left(1+\frac{\tilde{a}c}{kd^2}\right)} \left[ \left(\tilde{a}+ \frac{\tilde{a}^2 c^2}{kd^2}\right)   - \sqrt{\tilde{a}^2+ \frac{\tilde{a}^2c^4}{k^2d^4}+ \left( \frac{4}{k^2}-\frac{2}{k}\right) \frac{\tilde{a}^2 c^2}{d^2} } \right] \nonumber
\\ &=& \frac{\tilde{a}}{1+ \frac{\tilde{a}c}{kd^2}} \frac{\left( \frac{4}{k} - \frac{4}{k^2} \right) \frac{c^2}{d^2}}{ 2 \left[  \left( 1+ \frac{c^2}{kd^2} \right) + \sqrt{1+\frac{c^4}{k^2d^4} + \left( \frac{4}{k^2} - \frac{2}{k} \right) \frac{c^2}{d^2}  }     \right]}\,. \label{eqn:lamba-001}
\end{eqnarray}
Several things need to be checked. First we note that the term inside the square root is always nonnegative regardless of the choice of $k$ thanks to the form (\ref{eqn:1017}) and the positivity of $b$ and $d$. Next we see that 
\begin{equation*} 
\tilde{a} =  a- 2\eps_0 d = \frac{kad^2}{2dc +kd^2}  = \frac{kad}{2c + kd}>0\,,
\end{equation*}
thus as long as $k>1$, $\lambda(\eps_0)$ in (\ref{eqn:lamba-001}) is positive. Thirdly, we need to check that $\eps_0$ chosen in (\ref{eqn:eps0-2}) is strictly less than 1 so that the constant $C$ in (\ref{eqn:c-const}) is well defined. Let us choose 
\begin{equation} \label{eqn:k}
k = \frac{k_0ac}{d^2}>1\,, \quad k_0 \geq 2\,, 
\end{equation} 
then $\eps_0$ in (\ref{eqn:eps0-2}) becomes
\begin{equation} \label{eqn:eps0-3}
\eps_0 = \frac{ac}{2dc + k_0ac} < \frac{1}{k_0} <1\,.
\end{equation}
Plugging (\ref{eqn:eps0-3}) into (\ref{eqn:lamba-001}), one obtains 
\begin{eqnarray*}
\lambda (\eps_0) \geq \frac{\tilde{a}}{1+ \frac{\tilde{a} }{k_0 a}} \frac{\left(\frac{1}{k}-\frac{1}{k^2}\right) \frac{c^2}{d^2}}{1+\frac{c^2}{d^2}\frac{1}{k}} 
= \frac{\tilde{a}}{1+ \frac{\tilde{a} }{k_0 a}} \frac{(k_0 ac - d^2)}{(k_0 a + c)k_0 a}\,.
\end{eqnarray*}
Since $k_0 \geq \frac{2d^2}{ac}$ according to (\ref{eqn:k0}), we have 
\begin{equation*}
\lambda (\eps_0) \geq \frac{\tilde{a}}{1+ \frac{\tilde{a} }{k_0 a}} \frac{d^2}{(k_0 a+ c)k_0a}\,,
\end{equation*}
and therefore $\lambda \geq \lambda (\eps_0)$, which ends the proof. 

\end{proof}

In light of the previous two lemmas, we are ready to show the convergence rate $\lambda_\Kn$ in terms of $\Kn$.
\begin{theorem} \label{thm:lambda}
Denote $\lambda_\Kn$ the solution to the max min problem~\eqref{eqn:lambda}, defined by $\alpha$, $\beta$ and $\gamma$ rescaled by $\Kn$ in the parabolic scaling. Then in the zero limit of $\Kn$, $\lambda \sim \mathcal{O} \left(1 \right)$. Moreover, $\eps_0 \sim \mathcal{O}(1)$.
\end{theorem}
\begin{proof}
By Lemma~\ref{lemma:rescale}, $\alpha_\Kn = \frac{\alpha}{\Kn^2}$, $\beta_\Kn = \frac{\beta}{\Kn^2}$ and $\gamma_\Kn\sim\frac{\gamma}{\Kn}$ in the zero limit of $\Kn$, therefore
\begin{equation} \label{eqn:adc-para}
a_{\Kn} \sim \mathcal{O}\left( \frac{1}{\Kn^2}\right), \quad d_{\Kn} \sim \mathcal{O}\left( \frac{1}{\Kn}\right), \quad c_\Kn \sim \mathcal{O}(1)\,.
\end{equation}
Then one sees that the choice of $k_0$ in (\ref{eqn:k0}) makes it always order one, i.e. $k_0 \sim \mathcal{O}(1)$. Thus $\tilde{a}_\Kn$ from (\ref{eqn:lambda-0000}) scales as $\tilde{a}_\Kn\sim \mathcal{O}\left(\frac{1}{\Kn^2}\right)$. Consequently, $\lambda_{\Kn}$ remains $\mathcal{O}(1)$ for arbitrarily small $\Kn$. In view of (\ref{eqn:eps0-3}), $(\eps_0)_\Kn \sim \mathcal{O}(1)$. 
\end{proof}

\subsection{High field scaling}
In the high field scaling, $\Top_\Kn\to  \frac{1}{\Kn}\Top$ and $\Lop_\Kn\to\frac{1}{\Kn}\Lop$, then according to the definition of $\alpha$, $\beta$ and $\gamma$ in~\eqref{eqn:mic_coercive},~\eqref{eqn:mac_coercive} and~\eqref{eqn:boundedness}, we have:

\begin{lemma}\label{lemma:rescale-2}
In the high field regime (\ref{eqn:000}), we have
\begin{equation*}
\alpha_\Kn =\frac{\alpha}{\Kn}\,,\quad \beta_\Kn =\frac{\beta}{\Kn^2}\,,\quad \gamma_\Kn  = \gamma\,,
\end{equation*}
where $\alpha$, $\beta$ and $\gamma$ are the parameters when $\Kn = 1$.
\end{lemma}
The proof is similar to that for Lemma~\ref{lemma:rescale} and we omit it here. 

Next, we turn our attention to $\lambda$ in (\ref{eqn:lambda-5}) again. In the following we give a different lower bound of $\lambda$ from Lemma~\ref{lemma:lambda} to serve the high field rescaling later on. 

\begin{lemma}\label{lemma:lambda-6}
Under the same notation as in (\ref{eqn:abcd}), and let
\begin{equation}\label{eqn:eps0}
\eps_0  = \min \left\{ \frac{1}{2}, \quad  \frac{ac}{2d(c+d)} \right\}\,,
\end{equation}
then 
$\lambda$ to the max min problem \eqref{eqn:lambda-5} has a lower bound
\begin{equation}\label{eqn:lambda-1000}
\lambda \geq \lambda(\eps_0) = \left\{ \begin{array}{cc} 
\frac{\tilde{a}}{1+\frac{\tilde{a}c}{2d^2}} \frac{\frac{c^2}{d^2}}{2\left[ \left( 1+ \frac{c^2}{2d^2}\right) +  \sqrt{1+\frac{c^4}{4d^4}} \right]} & \text{ if } ~ \frac{\tilde{a}c}{d^2} \leq 1 
\\  \frac{1}{3} \frac{d^2}{ (a-d+\frac{1}{2}c) + \sqrt{ (a-d-\frac{1}{2}c)^2 + d^2 }  } & \text{ if }~ \frac{\tilde{a}c}{d^2} > 1 
\end{array}  \right. \,,
\end{equation} 
where 
\begin{equation} \label{eqn:a-tilde-0}
\tilde{a} = \frac{ad}{c + d}.
\end{equation}
\end{lemma}

\begin{proof}
Similar to Lemma~\ref{lemma:lambda}, we have $\lambda(\eps)$, the solution to (\ref{eqn:lambda}) as
\begin{equation*} \label{eqn:1120}
\lambda (\eps) =  \frac{1}{2}\left[ (a-2b +\eps c ) - \sqrt{(a-2b-\eps c )^2+4bd\eps} \right] \frac{1}{1+\eps} \,,
\end{equation*}
see equation (\ref{eqn:1017}). Let $\tilde{a} = a - 2b$, then we need 
\begin{equation}
\eps \leq \frac{\tilde{a}c}{d^2} 
\end{equation}
to make $\lambda(\eps)>0$. We also need $\eps<1$ for $C(\eps)$ in \eqref{eqn:c-const}, thus without lost of generality, we pick
\begin{equation} \label{eqn:eps0-5}
\eps_0 = \min\left\{ \frac{1}{2}, \quad \frac{\tilde{a}c}{2d^2}\right\}\,.
\end{equation}
Note that when $\frac{\tilde{a}c}{d^2}<1$, then $\eps_0 = \frac{\tilde{a}c}{2d^2}$, and we have 
\begin{equation} \label{eqn:eps0-7}
\eps_0 = \frac{ac}{(1+\gamma)c + 2d^2}\,,
\end{equation}
which can be obtained from (\ref{eqn:eps0-2}) by setting $k=2$. Therefore, (\ref{eqn:eps0-5}) reduces to (\ref{eqn:eps0}). Using such $\eps_0$, $\tilde{a}$ takes the form (\ref{eqn:a-tilde-0}). Then one just need to carry out the calculation of $\lambda\left(\frac{1}{2}\right)$ and $\lambda\left( \frac{\tilde{a}c}{2d^2} \right)$ to get (\ref{eqn:lambda-1000}). Note also that when $\frac{\tilde{a}c}{d^2}>1$, we have
\begin{eqnarray} \label{eqn:1201}
\lambda\left( \frac{1}{2} \right) = \frac{1}{3} \left[ (a-d+\frac{1}{2}c) - \sqrt{ (a-d-\frac{1}{2}c)^2 + d^2 }\right] 
= \frac{1}{3}\frac{2c(a-d)-d^2}{ (a-d+\frac{1}{2}c) + \sqrt{ (a-d-\frac{1}{2}c)^2 + d^2 }  }\,,
\end{eqnarray}
and since
\begin{equation} \label{eqn:1202}
\frac{\tilde{a}c}{d^2} = \frac{ac}{cd+d^2}>1\,.
\end{equation}
(\ref{eqn:1201}) becomes
\begin{equation} \label{eqn:1203}
\lambda\left( \frac{1}{2} \right) \geq \frac{1}{3}\frac{d^2}{ (a-d+\frac{1}{2}c) + \sqrt{ (a-d-\frac{1}{2}c)^2 + d^2 }  }\,.
\end{equation}
Equation (\ref{eqn:1202}) also implies that $a-d>0$ and thus the denominator in (\ref{eqn:1203}) is also positive. 
\end{proof}

Equipped with these two lemmas, we can similarly show the lower bound of $\lambda$ in the presence of $\Kn$ for the high field scaling.
\begin{theorem} \label{thm:lambda-6}
Denote $\lambda_\Kn$ the solution to the min max problem~\eqref{eqn:lambda}, defined by $\alpha$, $\beta$ and $\gamma$ rescaled by $\Kn$ in the high field scaling. Then in the zero limit of $\Kn$, $\lambda \sim \mathcal{O} \left(1 \right)$. Moreover, $\eps_0 \sim \mathcal{O}(1)$.
\end{theorem}
\begin{proof}
From Lemma~\ref{lemma:rescale-2} and the definition of $a$, $c$, $d$ in (\ref{eqn:abcd}), we immediately get 
\begin{equation} \label{eqn:adc-high}
a_\Kn \sim \mathcal{O}\left( \frac{1}{\Kn}\right), \quad d_\Kn \sim \mathcal{O}(1), \quad c_\Kn \sim \mathcal{O}(1)\,.
\end{equation}
Note that this is different from the parabolic scaling (\ref{eqn:adc-para}). One sees that in the zero limit of $\Kn$, the choice of $\eps_0$ in (\ref{eqn:eps0-7}) becomes infinity, and we use use $\eps_0 = \frac{1}{2}$. Consequently, we use the second bound in (\ref{eqn:lambda-1000}). Note that using the rescaling (\ref{eqn:adc-high}), this bound remains $\mathcal{O}(1)$.
\end{proof}
\begin{remark}
We emphasize that both lower bounds in Lemma~\ref{lemma:lambda} and Lemma~\ref{lemma:lambda-6} hold true in the generic cases but we separate the discussions purely for the ease of the scalings they are used for. If we stick to the bound provided by Lemma~\ref{lemma:lambda} in the high field regime,~\eqref{eqn:lambda-0000} will provide vanished $\lambda$ upon rescaling, which suggests no decay in time.
\end{remark}

\section{Incorporating the randomness: regularity result for $\Kn = 1$}
In this section we study how the randomness propagates in the solution when only one scale appears (i.e., $\Kn = 1$). The randomness comes into the equation through the collision operator $L_z$ and through initial data:
\begin{equation} \label{eqn:IC}
f(0,x,v,z) = f_0(x,v,z)\,.
\end{equation}
In neither $L_z$ and $f_0$ do we specify the distribution or the dependence on $z$. The question to address in this section is: given the smooth dependence of the collision operator and the initial data on $z$, does $f$ enjoy similarly good regularity?

Both the stochastic collocation method, and the generalized polynomial chaos method are spectral type of methods, and thus inherit all the properties: they provide high order of accuracy if and only if the solution indeed embraces high level of regularity. Facing such problems, it is standard for us to check the derivatives and show the boundedness or even the decay in time. More specifically, let $g_l$ denote the $l^{\text{th}}$ derivative in $z$ of $f$ as
\begin{equation} \label{eqn:gl0}
g_l = \frac{\partial^l}{\partial z^l} f\,,
\end{equation}
then for a fixed point $z_0 \in \Omega$ and all $g_l$ evaluated at $z_0$, $f$ writes as
\begin{equation}\label{eqn:f_analytic}
f(z) = \sum_{l=0}^\infty \frac{g_l}{l!} (z-z_0)^l\,.
\end{equation} 
To make sense of it, the series need to converge. That is, the convergence radius, which is defined by:
\begin{equation} \label{eqn:radius}
r(z_0) = \frac{1}{\limsup_{l\to\infty}\left(g_l(z_0)/l!\right)^{1/l}} \,, 
\end{equation}
should be uniformly bounded from below for all $z_0$. This essentially requires certain decay of $g_l(z_0)$ in $l$ uniformly in $z_0$. The norm we use to measure the decay is the norm we have for the convergence in~\eqref{eqn:f_analytic}. In this paper, we show the decay of $g_l$ in $l$ in $L_\infty(t,L_2(\rd{\mu}))$, with which norm we make sense of the series in~\eqref{eqn:f_analytic}. Moreover we will show the decay in time is exponential with a rate independent of $\Kn$, uniformly bounded from below.

We derive the equation for $g_l$ and study its boundedness first, and two special cases of $L_z$ will be handled afterwards.

\begin{remark}
The best results one could hope for should be done point-wisely in time, space and velocity, then~\eqref{eqn:f_analytic} makes sense in a point-wise fashion. To date, there has been no literature that addresses such type of convergence to our knowledge, although it is predictable in certain cases. We leave that to future research. We also note that with the Galerkin framework, termed $\text{P}_N$ method for the transport equation specifically, the convergence in $L_2$ typically suffices.
\end{remark}

\subsection{Strategy of proof} 
To begin with, we assert that with $\alpha$, $\beta$ and $\gamma$ bounded above and below, the decay rate $\lambda_z$ has a lower bound as well. More specifically, we assume that Assumptions 1--4 hold true point-wisely in $z$, and therefore denote the constants therein by $\alpha_z$, $\beta_z$ and $\gamma_z$ to elucidate such dependence. We also assume that these constants are uniformly bounded from above and below for all $z$ under consideration, i.e., 
\begin{equation} \label{assumption:z}
0< \underline{\alpha} \leq \alpha_z \leq \bar{\alpha} < \infty, \quad
0< \underline{\beta} \leq \beta_z \leq \bar{\beta} < \infty, \quad
0< \underline{\gamma} \leq \gamma_z \leq \bar{\gamma} < \infty,  \quad \forall z\,.
\end{equation}
Then a combination of Lemma~\ref{lemma:lambda} and Lemma~\ref{lemma:lambda-6} lead to 
\begin{lemma}\label{lemma:lambda-z}
$\lambda_z$, defined as
\begin{equation} \label{eqn:lambdazzz}
\lambda_z =  
 \max_{\varepsilon_z, \delta_z } \min 
\left\{\frac{\alpha  - \varepsilon_z(1+\gamma_z) \left( 1+ \frac{1}{2\delta_z}\right) }{1+\eps_z}\,, ~\frac{\varepsilon_z}{1+\eps_z} \left[\frac{\beta_z}{1+\beta_z}-(1+\gamma_z)\frac{\delta_z}{2} \right]\right\}\,
\end{equation}
has a lower bound:
\begin{equation} \label{eqn:lambdaz000}
\lambda_z \geq \lambda_z(\eps_{z,0}) = \max \left\{  \frac{\tilde{a}_zd_z^2}{(k_{z,0} a_z + \tilde{a}_z) (k_{z,0} a + c)} , \quad 
\frac{1}{3} \left[ \left(a_z-d_z+\frac{1}{2}c_z \right) - \sqrt{ \left(a_z-d_z-\frac{1}{2}c_z \right)^2 + d_z^2 }\right]
\right\}\,,
\end{equation}
where
\begin{equation*}
\eps_{z,0} = \min\left\{ \frac{1}{2},\quad \frac{a_zc_z}{2d_zc_z + k_zd_z^2} \right\}\,,
\end{equation*}
and
\begin{equation*}
\tilde{a}_z = \frac{k_{z,0} a_z^2}{k_{z,0} ac + 2d}, \quad k_{z,0} = \max \left\{ 2, ~\frac{2d_z^2}{a_zc_z} \right\}, \quad k_z = \frac{k_{z,0} a_zc_z}{d_z^2}\,.
\end{equation*}
Here $a_z$, $c_z$, $d_z$ are defined the same as in (\ref{eqn:abcd}) but with a subscript $z$ to indicate the $z$-dependence. 
\end{lemma}
\begin{proof}
This theorem is a simple combination of Lemma~\ref{lemma:lambda} and \ref{lemma:lambda-6}, and we omit the proof here. 
\end{proof}

Then it is immediate that, under the assumption (\ref{assumption:z}), we have
\begin{corollary}
$\lambda_z$ defined in (\ref{eqn:lambdazzz}) has a lower bound that is strictly away from zero, i.e., 
\begin{equation*}
\lambda_z(\eps_{z,0}) \geq \underline{\lambda} > 0\,.
\end{equation*}
\end{corollary}
\begin{proof}
Note that $\lambda_z(\eps_{z,0})$ in (\ref{eqn:lambdaz000}) is strictly greater than zeros for any $z\in \Omega$, therefore, we can choose $\underline{\lambda} = \min_z \lambda_z(\eps_{z,0}) >0$. 
\end{proof}

This corollary allows us to show that given the initial data  $\int \int f_0\rd{x}\rd{v} =0$, the decay of $\sup_z \|f\|$ towards 0 is exponentially fast with a non-zero decay rate. It also prepares the ingredient that assist in showing the exponential decay in time in later part of this section.

The strategy in showing the regularity lies in bounding $g_l$ under some norm. It is not immediate since the brute force analysis gives the factorial growth of $g_l$ in $l$. To better illustrate the idea, we first consider a simpler case with
\begin{equation} \label{eqn:Lz}
\Lop_z = \sigma(z,x) \Lop \,,
\end{equation}
where $\Lop$ is the deterministic operator considered in the previous section. More general interaction between the randomness and collision will be considered in section \ref{sec:general-Kn-1}. Recall the kinetic equation
\begin{equation} \label{eqn:fagain}
\partial_t f + \Top  f = \Lop_z f = \sigma(z,x) \Lop f \,,
\end{equation}
we write down the equation satisfied by its $l^\text{th}$ derivative ($g_l$ defined in~\eqref{eqn:gl0}):
\begin{equation} \label{eqn:gl}
\partial_t g_l + \Top g_l = \Lop_z g_l + \sum_{k=0}^{l-1} \frac{l!}{k!(l-k)!} \partial_z^{l-k} \sigma \Lop g_k\,.
\end{equation}
We would like to adopt the techniques that shows the hypocoercivity of the original equation, but the equation for $g_l$, compared with the one for $f$ has an extra source term. What is more, the source term is essentially a combination of the previous $g_k$ (with $k<l$), and the influence of the randomness propagates along the chain in a combinatorics fashion. Without careful study of the structure of the equations, such effects blows up extremely fast as $l$ increases. The goal of this section is to provide new and sharper estimates that addresses the dependence on the source term, and control the growth of the random effects.

Following the proof of Theorem \ref{thm:hypocoercive}, we define the entropy of $g_l$
\begin{equation}\label{eqn:Hgl}
\Hfun[g_l] = \half \| g_l\| ^2 + \eps_z \average{ \Aop g_l, g_l }\,,
\end{equation}
where $\Aop$ is still defined as (\ref{eqn:Aop}). Then taking the derivative in $t$ of (\ref{eqn:Hgl}), we have, upon substituting (\ref{eqn:gl})
\begin{equation*}
\frac{d}{dt} \Hfun[g_l] = -\Dfun[g_l] + \average{\Source, g_l} + \eps_z \average{\Aop g_l, \Source} \,,
\end{equation*}
where $\Dfun$ takes the same form as in (\ref{eqn:D}), and $\Source$ is the source term:
\begin{equation} \label{eqn:Source}
\Source = \sum_{k=0}^{l-1} \frac{l!}{k!(l-k)!} \partial_z^{l-k} \sigma \Lop g_k\,.
\end{equation}
Then from Theorem \ref{thm:hypocoercive}, for every $z$, given a fixed $\eps_z$, we can estimate (\ref{eqn:Hgl}) as
\begin{equation} \label{eqn:1018}
\frac{d}{dt} \Hfun[g_l] \leq -2\lambda_z \Hfun[g_l] +  \average{\Source, g_l} + \eps_z \average{\Aop g_l, \Source}\,,
\end{equation}
with $\lambda_z$ defined in (\ref{eqn:lambdazzz}). Note that
\begin{equation*}
\average{\Aop g_l ,\Source} \leq \| \Aop g_l \| \| \Source\| \leq \half \| (I-\Pi) g_l \| \| \Source \|
\end{equation*}
thanks to (\ref{eqn:000}), and 
\begin{equation*}
\average{\Source, g_l } \leq  \| \Source \| \| g_l \| \,
\end{equation*}
we have
\begin{equation} \label{eqn:4.11}
\eps_z\average{\Aop g_l ,\Source}  + \average{\Source, g_l } \leq (1+\eps_z) \| g_l\| \|\Source\|\,,
\end{equation}
where we have used the fact that $\| \Pi\| \leq 1$. Then (\ref{eqn:1018}) can be further bounded by
\begin{equation} \label{eqn:Hglt}
\frac{d}{dt} \Hfun[g_l] \leq -2\lambda_z \Hfun[g_l] +  (1+\eps_z) \| g_l\| \| \Source\| \,.
\end{equation}
Given the form of $\Source$ in (\ref{eqn:Source}), we consider two cases in the following two subsections. 

\subsection{Case 1: $\sigma(z,x)$ has an affine dependence on $z$} \label{sec:case1}
In this case, we assume $\sigma$ linearly depends on $z$, therefore $\partial^l_z\sigma = 0$ for $l>1$. It is a standard example, especially when the randomness is extracted from the Karhunen-Lo\'{e}ve expansion~\cite{Loeve1}.
\begin{equation} \label{eqn:C1}
C_1 = \sup_{x}|\partial_z \sigma| \,
\end{equation}
Then $\Source$ reduces to
\begin{equation} \label{eqn:Source-1}
|\Source| = \left|\sum_{k=0}^{l-1} \frac{l!}{k!(l-k)!} \partial_z^{l-k} \sigma \Lop g_k \right|\leq  C_1 l  |\Lop g_{l-1}|\,,
\end{equation}
and (\ref{eqn:Hglt}) becomes
\begin{eqnarray}
\frac{d}{dt} \Hfun[g_l]  &\leq&  -2\lambda_z \Hfun[g_l] + (1+\eps_z) C_1 l \norm{\Lop g_{l-1}} \norm{g_l} \nonumber
\\ &\leq &  -2\lambda_z \Hfun[g_l] + (1+\eps_z) C_1 l \norm{ g_{l-1}} \norm{g_l}  \nonumber
\\ &\leq & -2\lambda_z \Hfun[g_l]  + C_1 l (1+\eps_z) \frac{2}{1-\eps_z} \sqrt{\Hfun[g_{l}]} \sqrt{\Hfun[g_{l-1}]}\,.  \label{eqn:420}
\end{eqnarray}
Here the second inequality uses $\norm{\Lop} \leq 1$, and the third one uses the relation between $\norm{g_l}$ and entropy $\Hfun[g_l]$ in (\ref{eqn:relation1}). In fact, the bound for $\norm{L}$ can be relaxed to any constant, and we use $1$ just for brevity of notation. Notice that (\ref{eqn:420}) is equivalent to
\begin{equation*}
\frac{d}{dt} \left( \sqrt{\Hfun[g_l]}\right)^2  \leq -2\lambda_z \Hfun[g_l]  + 2C_1 l \frac{1+\eps_z}{1-\eps_z} \sqrt{\Hfun[g_l]} \sqrt{\Hfun[g_{l-1}]} \,,
\end{equation*}
which readily implies
\begin{equation} \label{eqn:sHglt}
\frac{d}{dt} \sqrt{\Hfun[g_l]}  \leq -\lambda_z \sqrt{\Hfun[g_l]} + \tilde{C}_1 l \sqrt{\Hfun[g_{l-1}]}, \qquad \Cto = C_1 C_z^2 , \quad  C_z =  \frac{1+\eps_z}{1-\eps_z} \,.
\end{equation}
Note that $g_0 = f$, we have
\begin{equation} \label{eqn:sHg0t}
\frac{d}{dt} \sqrt{\Hfun[g_0]}  \leq -\lambda_z \sqrt{\Hfun[g_0]}  \,.
\end{equation}
Let 
\begin{equation}\label{eqn:hl}
h_l =  \sqrt{\Hfun[g_l]}\geq 0\,,
\end{equation}
and rewrite (\ref{eqn:sHglt}) into
\begin{equation} \label{eqn:hlt}
\frac{d}{dt} h_l \leq -\lambda_z h_l + \Cto l h_{l-1} \,,
\end{equation}
the we have the following lemma. The proof is based on mathematical induction and it is postponed to the appendix. 
\begin{lemma} \label{lemma:hl}
$h_l$ defined in (\ref{eqn:hl}) satisfies
\begin{equation} \label{eqn:hle}
h_l(t) \leq e^{-\lambda_z t}  \sum_{k=0}^l \frac{l!}{(l-k)!k!} (\Cto t)^k h_{l-k}(0)  \,,
\end{equation}
where $\lambda_z$ is defined in (\ref{eqn:lambdazzz}) and $h_j(0)$ represents the initial data. 
\end{lemma}

As a consequence, we have the following estimate for $g_l$.
\begin{theorem}\label{thm:gl1}
If we assume that the initial data (\ref{eqn:IC}) satisfies 
\begin{equation} \label{eqn:ICaaa}
\norm{\partial_z^l f_0(z)} = \norm{g_l(0)} \leq H^l, \qquad \text{for all } \ l \geq 0
\end{equation}
$g_l$, the $l^\text{th}$ derivative of $f$ in $z$ can be estimated as
\begin{equation*}
\norm{g_l} \leq C_z e^{-\lambda_z t} (H+ t\Cto)^l\,,
\end{equation*}
with $C_z$ defined in~\eqref{eqn:sHglt}.
\end{theorem}

\begin{remark}
There are two immediate take away information:
\begin{itemize}
\item[1.] Long time behavior: it is obvious that as $t\rightarrow \infty$, $\norm{g_l} \rightarrow 0$, as one would expect.
\item[2.] Convergence radius: as mentioned in \eqref{eqn:radius}, the convergence radius for $f$ at any point $z_0$ is 
\begin{equation}
r(z_0) = \frac{1}{\limsup_{l\to\infty}\left(g_l(z_0)/l!\right)^{1/l}} =\infty\,,
\end{equation}
which is independent of $z_0$, and thus $f$ is {\it analytic} in $z$. Note also that this radius is independent of $\lambda_z$, which implies that the analyticity of $f$ is irrelevant to its long time behavior.
\end{itemize}
\end{remark}

\subsection{Case 2: $\sigma(z,x)$ has an arbitrary dependence on $z$ with $\left|  \frac{\partial_z^n \sigma }{ n!}\right| \leq C_2$} \label{sec:case2}
Now we move on to a more general case where the dependence of $\sigma(x,z)$ on $z$ can be arbitrary. The only condition we impose here is that $\left|\frac{1}{n!}\frac{\rd^n\sigma}{\rd z^n}\right|<C_2$ for all $n$, where $C_2$ is a constant. This is in fact a very relaxed condition: it allows the $n$-th derivative growing as $n!$. It can hardly be loosen anymore since $\sigma$ itself needed to be a well-defined function, having nontrivial convergence radius.

Under this assumption, (\ref{eqn:Hglt}) rewrites to 
\begin{equation}\label{eqn:119}
\frac{d}{dt} \Hfun[g_l] \leq -2\lambda_z \Hfun[g_l] +  (1+\eps_z) C_2 \| g_l\|  \sum_{k=0}^{l-1} \frac{l!}{k!} \norm{g_k}  \,.
\end{equation}
To lighten the notation that needed in the following calculations, we let
\begin{equation} \label{eqn:notation}
\tilde{g}_l = \frac{g_l}{l!}, \quad \Hfun[\tilde{g}_l] = \frac{1}{(l!)^2} \Hfun[g_l], \quad \tilde{h}_l = \sqrt{\Hfun[\tilde{g}_l]} \,, \quad
\eta_l = e^{\lambda_z t} \tildeh_l\,.
\end{equation}
Dividing (\ref{eqn:119}) by $(l!)^2$ on both sides, we have
\begin{equation} \label{eqn:120}
\frac{d}{dt} \Hfun[\tilde{g}_l] \leq -2\lambda_z \Hfun[\tilde{g}_l] + (1+\eps_z) C_2 \norm{\tilde{g}_l} \sum_{k=0}^{l-1} \norm{\tilde{g}_k}  \,.
\end{equation}
Using the notion in (\ref{eqn:notation}) and together with the relation (\ref{eqn:relation1}), we get
\begin{equation} \label{eqn:C2tilde}
\frac{d}{dt} \tilde{h}_l \leq -\lambda_z \tilde{h}_l + \Ctt \sum_{k=0}^{l-1} \tildeh_k\,, \qquad \Ctt = C_2 C_z^2 = C_2 \frac{1+\eps_z}{1-\eps_z}\,, l \geq 1\,.
\end{equation}
Similar to (\ref{eqn:sHg0t}), we have 
\begin{equation} \label{eqn:121}
\frac{d}{dt} \tildeh_0 \leq -\lambda_z \tildeh_0\,.
\end{equation}
With the relation between $\eta_l$ and $\tildeh_l$ in (\ref{eqn:notation}), we further reduce (\ref{eqn:C2tilde}) (\ref{eqn:121}) to
\begin{equation} \label{eqn:eta000}
\frac{d}{dt} \eta_l \leq \Ctt \sum_{k=0}^{l-1} \eta_k\,,\quad \text{for}~ l \geq 1\,,\quad\text{and}  \qquad \frac{d}{dt} \eta_0 \leq 0\,.
\end{equation}
Now it amounts to estimate $\eta_l$ and we have the following lemma.
\begin{lemma}\label{lem:etal}
Assuming the initial condition for $\eta_l$ satisfy:
\begin{equation} \label{eqn:etaIC}
\norm{\eta_l(0)} \leq \frac{H^l}{l!} \,, \quad \text{ for all } ~ l \geq 0, 
\end{equation}
then the solution to~\eqref{eqn:eta000} satisfy:
\begin{equation}\label{eqn:bound_eta_case2_sharp}
\norm{\eta_l} \leq   \frac{H^l}{l!} +  \sum_{k-1}^l \frac{(\tilde{C}_2t)^k}{k! (k-1)!} \frac{(l-1)!}{(l-k)!} (1+H)^{l+1} \,,
\end{equation}
and it could be further relaxed to:
\begin{equation}\label{eqn:bound_eta_case2}
\norm{\eta_l}\leq \frac{H^l}{l!} +  (1+H)^{l+1}  \min\{  (1+C_2 t)^l \,, \quad e^{\Ctt t} 2^{l-1}  \}\,.
\end{equation}
\end{lemma}
The proof needs a long detailed calculation of the solution to (\ref{eqn:eta000}) and we leave it in the appendix not to distract the reader. Getting back to $\tildeg_l = \frac{g_l}{l!}$, we have:
\begin{theorem} \label{thm:gl2}
Under the assumption of initial condition (\ref{eqn:ICaaa}), we have
\begin{equation} 
\left\|\frac{g_l}{l!}\right\| \leq \sqrt{\frac{2}{1-\eps_z}} \frac{H^l}{l!} e^{-\lambda_z t} + \sqrt{\frac{2}{1-\eps_z} } (1+H)^{l+1} \min   \{ e^{-\lambda_z t}(1+\Ctt t)^l , \quad  e^{ (\Ctt-\lambda_z) t}  2^{l-1}\} \,.
\end{equation}
\end{theorem}

\begin{remark}
The bound~\eqref{eqn:bound_eta_case2} is far from being sharp but~\eqref{eqn:bound_eta_case2_sharp} is. The loosen bound, which gets translated into Theorem~\ref{thm:gl2} has two terms and they are used for different purposes.
\begin{itemize}
\item[1.] Long time behavior: to understand the long time behavior, we look at the first bound. For every fixed $n$, as $t \rightarrow \infty$, $t^l e^{-\lambda_z}$ is dominated by the exponential function and given that $\lambda_z $ strictly less than 0, all derivatives decay exponentially in the long time limit. 

\item[2.] Convergence radius: here we use the second bound. For every fixed $t$, $e^{(\Ctt-2\lambda_z)t}$ does not play a role and by definition (\ref{eqn:radius}) the convergence radius 
\begin{equation}
r(z_0) = \frac{1}{\limsup_{l\to\infty}\left(e^{ (\Ctt-\lambda_z) t/l}  2^{(l-1)/l} (1+H)^{(l+1)/l}\right)} =\frac{1}{2(1+H)}\,,
\end{equation}
which is independent of $z_0$.
\end{itemize}
\end{remark}

Putting the above results together, we have the following conclusion.
\begin{theorem} \label{thm:ana}
Consider the initial value problem (\ref{eqn:f}) (\ref{eqn:IC}). If we assume that $\Lop_z = \sigma(x,z) \Lop$ with $\Lop$ being deterministic, and let the following two conditions 
\[
\norm{\frac{\partial_z^l \sigma }{ l!} }\leq C_2, \qquad  \norm{\partial_z^l f_0} \leq H^l 
\]
be satisfied for all integers $l\geq 0$, then the solution $f(t,x,v,z)$ to the initial value problem is analytic in any $z_0 \in \Omega$ with uniform convergence radius $\frac{1}{2(1+H)}$. Moreover, all the derivatives of $f$ in $z$ decays exponentially in time. Here the norm is induced by the inner product in \eqref{eqn:322-2}\,.
\end{theorem}

\subsection{General form of $\Lop_z$} \label{sec:general-Kn-1}
We would like to mention briefly in this subsection that all the computations above can be extended to more general case where the randomness in the collision operator can be more involved than (\ref{eqn:Lz}). Let us take the anisotropic scattering operator (\ref{eqn:1.7}) for example. As always, we take the $l^{\text{th}}$ derivative in $z$ of 
\begin{equation*}
\partial_t f  + \Top f = \Lop_z f
\end{equation*}
to get 
\begin{equation} \label{eqn:4.36}
\partial_t g_l + \Top g_l = \Lop_z g_l + \sum_{k=0}^{l-1} \frac{l!}{k!(l-k)!}  \int \left[ \partial_z^{l-k} \Kop_z(v^* \rightarrow v) g_k(v^*) - \partial_z^{l-k} \Kop_z(v \rightarrow v^*) g_k(v) \right] dv\,.
\end{equation}
Here $g_l$ is defined the same as in (\ref{eqn:gl0}). Denote the operator 
\begin{equation*}
\Lop^{q}_z f = \int \left[ \partial_z^{q} \Kop_z(v^* \rightarrow v) f(v^*) - \partial_z^{q} \Kop_z(v \rightarrow v^*) f(v) \right] dv\,,
\end{equation*}
then (\ref{eqn:4.36}) is compressed to 
\begin{equation} \label{eqn:4.37}
\partial_t g_l + \Top g_l = \Lop_z g_l  + \sum_{k=0}^{l-1} \frac{l!}{k!(l-k)!}  \Lop_z^{l-k} g_k \,.
\end{equation}
Compare it to (\ref{eqn:gl}), we see that as long as 
\begin{equation}
\left\| \Lop_z^{q} f \right\| \leq C_L \norm{f}, \quad  \forall ~{\text{ integer} }~ q \,,
\end{equation}
(\ref{eqn:4.37}) boils down to exactly the same problem as before.

\section{Incorporating the randomness: regularity result for $\Kn \ll 1$}
Equipped with previous estimates for random case with $\Kn = 1$ and deterministic case with $\Kn \ll 1$, we can directly adapted them to the case with much smaller $\Kn \ll 1$ and with randomness. First we emphasis that for each individual $z$, the lower bound of $\lambda_z$ obtained in Lemma \ref{lemma:lambda-z} remains $\mathcal{O}(1)$ for $\Kn\ll1$ thanks to Theorem \ref{thm:lambda} and \ref{thm:lambda-6}. Then one just need to take a minimum over all $z\in \Omega$ to get a uniform lower bound. Therefore, a decay in time of $g_l$ is out of question. 

Below we will address the convergence radius of (\ref{eqn:f_analytic}) with two different scaling separately. As mentioned in Section \ref{sec:general-Kn-1}, the more general collision operator (anisotropic for example) can be treated in exactly the same way, our discussion below will be centered on the case with $\Lop_z = \sigma(x,z) \Lop$. 

\subsection{Parabolic scaling} \label{sec:para000}
In the parabolic scaling, consider
\begin{equation} \label{eqn:fKn}
\partial_t f + \frac{1}{\Kn}\Top f = \frac{1}{\Kn^2} \sigma(x,z) \Lop f
\end{equation}
the following the same procedure as in (\ref{eqn:gl}) -- (\ref{eqn:4.11}), we arrive at 
\begin{equation} \label{eqn:HglKn}
\frac{d}{dt} \Hfun[g_l] \leq -2\lambda_{z,\Kn} \Hfun[g_l] + (1+\eps_z) \norm{g_l} \norm{\Source}, \qquad \Source =  \frac{1}{\Kn^2} \sum_{k=0}^{l-1} \frac{l!}{k!(l-k)!} \partial_z^{l-k} \sigma \Lop g_k\,,
\end{equation}
which is similar to (\ref{eqn:Hglt}), but with $\lambda_z$ replaced by $\lambda_{z,\Kn}$, and the source amplified by $\frac{1}{\Kn^2}$. The former change will not introduce any difference as already shown in Section 3 that the including of small $\Kn$ won't diminish $\lambda_z$. The latter change plays the role of enlarging the constants $C_1$ and $C_2$ in Sections \ref{sec:case1} and \ref{sec:case2} by $\frac{1}{\Kn^2}$. We will see in the following that this change will not affect the regularity of $f$ in $z$. 

\vspace{0.3cm}
\hspace{-0.45cm} {\bf Case 1: $\sigma(x,z)$ has an affine dependence on $z$} 
\\As written in (\ref{eqn:HglKn}), the amplification in $\Source$ results in the same effect for $C_1$ in (\ref{eqn:C1}) and $\Cto$ in (\ref{eqn:sHglt}) as well. Therefore, we restate Theorem \ref{thm:gl1} here to add the $\Kn$ dependence.
\begin{theorem} \label{thm:para-Kn-z}
If we assume that 
\begin{equation*}
|\partial_z\sigma(x,z)| = C_1, \quad \text{and} \quad \partial_z^l \sigma(x,z) \equiv 0 \quad {\text for} ~l \geq 2\,,
\end{equation*}
and initial data still satisfies (\ref{eqn:ICaaa}), then the $l^\text{th}$ derivative of the solution to (\ref{eqn:fKn}), denoted by $g_l$ has the following estimate
\begin{equation} \label{eqn:4.43}
\norm{g_l} \leq C e^{-\lambda_{z,\Kn}t}  \left(H+t \frac{\Cto}{\Kn^2} \right)^l\,.
\end{equation}
\end{theorem}
It is easy to see that, even in the presence of $\Kn$ in (\ref{eqn:4.43}), we still have 1) Exponential decay in time for all derivatives $\norm{g_l}$; 2) Infinite convergence radius for any $z_0\in \Omega$. 

\vspace{0.3cm}
\hspace{-0.45cm} {\bf Case 2: $\sigma(x,z)$ has an arbitrary dependence on $z$} 
\\As in the previous case, the diffusive scaling only enlarges $\Ctt$ by $\frac{1}{\Kn^2}$ while keeping all the derivation still valid. Therefore, Theorem~{\ref{thm:gl2}} still holds with $\Ctt$ replaced by $\frac{1}{\Kn^2} \Ctt$. More precisely, we have 
\begin{theorem} \label{thm:para-Kn-z2}
If we assume that 
\[
\left|\frac{\partial_z^l \sigma }{ l!} \right|\leq C_2, 
\]
and initial data satisfies (\ref{eqn:ICaaa}), then
\begin{equation*} 
\left\|\frac{g_l}{l!}\right\| \leq \sqrt{\frac{2}{1-\eps_z}} \frac{H^l}{l!} e^{-\lambda_{z,\Kn} t} + \sqrt{\frac{2}{1-\eps_z}} (1+H)^{l+1} \min \left\{ e^{-\lambda_{z,\Kn} t} \left(1+ \frac{\Ctt}{\Kn^2} t \right)^l , \quad   e^{ \left(\frac{\Ctt}{\Kn^2}-\lambda_{z,\Kn} \right) t}  2^{l-1} \right\} \,.
\end{equation*} 
\end{theorem}
Then again the former term in the bound guarantees the long time exponential decay and the latter one governs the analyticity of $f$. 

\subsection{High field scaling}
This section can be considered as a duplication of Sections \ref{sec:para000} with a slight variation by changing $\Kn$ dependence to the high field regime. In particular, recall the problem we consider
\begin{equation} \label{eqn:fKn-high}
\partial_t f + \frac{1}{\Kn} \Top f = \frac{1}{\Kn} \sigma(x,z) \Lop f\,,
\end{equation}
the parallel to (\ref{eqn:HglKn}), we have
\begin{equation} \label{eqn:HglKn-high}
\frac{d}{dt} \Hfun[g_l] \leq -2\lambda_{z,\Kn} \Hfun[g_l] + (1+\eps_z) \norm{g_l} \norm{\Source}, \qquad \Source =  \frac{1}{\Kn} \sum_{k=0}^{l-1} \frac{l!}{k!(l-k)!} \partial_z^{l-k} \sigma \Lop g_k\,.
\end{equation}
Here we use the same notation $\lambda_{z,\Kn}$, but it is different from that in (\ref{eqn:HglKn}), yet still strictly bounded away from zero for arbitrarily small $\Kn$ thanks to Theorem \ref{thm:lambda-6}. The source term, as opposed to (\ref{eqn:HglKn}), is only amplified by $\frac{1}{\Kn}$. Consequently, we have the following two theorems regarding the two cases, both of which enjoys an exponential decay in time and analyticity in the random space.

\hspace{-0.45cm} {\bf Case 1: $\sigma(x,z)$ has an affine dependence on $z$} 
\\Similar to Theorem \ref{thm:para-Kn-z}, we have
\begin{theorem}
If we assume that 
\begin{equation*}
|\partial_z\sigma(x,z)| = C_1, \quad \text{and} \quad \partial_z^l \sigma(x,z) \equiv 0 \quad {\text for} ~l \geq 2\,,
\end{equation*}
and initial data still satisfies (\ref{eqn:ICaaa}), then the $l-th$ derivative of the solution to (\ref{eqn:fKn-high}), denoted by $g_l$ has the following estimate
\begin{equation} \label{eqn:4.43}
\norm{g_l} \leq C e^{-\lambda_{z,\Kn}t}  \left(H+t \frac{\Cto}{\Kn} \right)^l\,.
\end{equation}
\end{theorem}

\vspace{0.5cm}
\hspace{-0.45cm} {\bf Case 2: $\sigma(x,z)$ has an arbitrary dependence on $z$} 
\\In this case, we have the following theorem that resembles Theorem \ref{thm:para-Kn-z2}.
\begin{theorem}\label{thm:high-322}
If we assume $\left|\frac{\partial_z^l \sigma }{ l!} \right|\leq C_2$, and that the initial data satisfies (\ref{eqn:ICaaa}), then
\begin{equation*} 
\left\|\frac{g_l}{l!}\right\| \leq \sqrt{\frac{2}{1-\eps_z}} \frac{H^l}{l!} e^{-\lambda_{z,\Kn} t} + \sqrt{\frac{2}{1-\eps_z}} (1+H)^{l+1} \min \left\{ e^{-\lambda_{z,\Kn} t} \left(1+ \frac{\Ctt}{\Kn} t \right)^l , \quad   e^{ \left(\frac{\Ctt}{\Kn}-\lambda_{z,\Kn} \right) t}  2^{l-1} \right\} \,.
\end{equation*} 
\end{theorem}

\section{Conclusion}
In this paper, we prove the uniform regularity results for linear multiscale kinetic equations with random input. Our proof builds on a general framework that can be applied to a wide range of linear kinetic equations, and to different regimes including kinetic, diffusive and high field. In the macroscopic scalings, a direct estimate reveals that the solution will lose regularity due to the stiffness exerted by the small parameter. However, we showed that, via a careful and sharp calculation of the high order derivatives of the solution in random variables, the solution remains analytic in the random space. Moreover, based on a hypocoercivity argument, we recover the exponential decay in time of any derivatives of the solution. This result is expected to play a key role in validating any spectral or pseudo-spectral based numerical methods for kinetic equations with multiple scales, such as stochastic Galerkin method and stochastic collocation method.

\begin{appendix}
\section{Proof of Lemma~\ref{lemma:hl}}
\begin{proof}
This lemma can be proved by mathematical induction. First rewrite (\ref{eqn:hlt}) into an integral form
\begin{equation}\label{eqn:hli}
h_l (t) \leq e^{-\lambda_z t} h_l(0) + \Cto l e^{-\lambda_z t} \int_0^t e^{\lambda_z s} h_{l-1}(s) ds\,.
\end{equation}
Note from (\ref{eqn:sHg0t}) that 
\begin{equation*}
h_0(t) \leq e^{-\lambda_z t} h_0(0)\,,
\end{equation*}
which along with (\ref{eqn:hli}) immediately implies that
\begin{equation*}
h_1(t) \leq e^{-\lambda_z t} \left[ h_1(0)  + \Cto t h_0(0)\right]\,.
\end{equation*}
Now we assume (\ref{eqn:hle}) holds for all $l$ up to $n$ and we need to verify it for $l=n+1$. Indeed, from (\ref{eqn:hli}), we have
\begin{eqnarray*}
h_{l+1}(t) &\leq& e^{-\lambda_z t} h_{l+1}(0) + \Cto (l+1) e^{-\lambda_z t} \int_0^t e^{-\lambda_z s} h_l(s)\rd{s}\nonumber
\\ & \leq &  e^{-\lambda_z t} h_{l+1}(0) + \Cto (l+1) e^{-\lambda_z t} \int_0^t \sum_{k=0}^l \frac{l!}{(l-k)!k!} (\Cto s)^k h_{l-k}(0)\rd{s} \nonumber
\\ &= & e^{-\lambda_z t} \left[ h_{l+1}(0) +  \sum_{k=0}^l \frac{(l+1)!}{(l-k)!(k+1)!} (\Cto s)^{k+1} h_{l-k}(0)  \right] \nonumber
\\ & = & e^{-\lambda_z t} \left[ h_{l+1}(0) +  \sum_{m=1}^{l+1} \frac{(l+1)!}{(l+1-m)!m!} (\Cto s)^{m} h_{l+1-m}(0)\right] \,,
\\ & = & e^{-\lambda_z t} \sum_{k=0}^{l+1} \frac{(l+1)!}{(l+1-k)!k!} (\Cto s)^{k} h_{l+1-k}(0)\,.
\end{eqnarray*}
which finishes the induction\,.
\end{proof}

\section{Proof of Lemma~\ref{lem:etal}}
Since $\eta_l$ is nonnegative, we can estimate it by calculating the solution to (\ref{eqn:eta000}) with an equal sign. Fix $l$, we rewrite the ODE system in a matrix form 
\begin{equation} \label{eqn:B1}
\frac{\rd}{\rd t}\vec{\eta} = \Ctt A\cdot\vec{\eta}\,
\end{equation}
where
\begin{equation}
\vec{\eta} = [\eta_l, \eta_{l-1}, \cdots \eta_0]^t
\end{equation} 
is an $(l+1)\times 1$ vector and 
\begin{equation}
A = \left(\begin{array}{cccccc}0 & 1 & 1 &\cdots &\cdots & 1\\ & 0 & 1 &\cdots &\cdots & 1\\& &\ddots &\ddots  & & 1\\&  &  & \ddots & & 1\\& &  & & & 0\end{array}\right)\,,
\end{equation}
is an $(l+1)\times (l+1)$ matrix. It is easy to check that $A$ has an eigenvalue $0$ with multiplicity $l+1$ and the associated eigenvector is $[1,0,\cdots, 0]^t$. Now we decompose it in the form of Jordan block
\begin{equation}
A\cdot S = S\cdot J\,,
\end{equation}
with
\begin{equation}
J = \left(\begin{array}{cccccc}0 &1 & & & & \\ & 0 & 1 & & & \\& & \ddots & \ddots & & \\&  &  &  & & 1\\& &  & & & 0\end{array}\right)\,,  \qquad
S = [S_0| S_1| S_2| \cdots |S_l]\,,
\end{equation}
Here 
\begin{equation}
A\cdot S_0 = 0\,,\quad \text{and}\quad   A\cdot S_{m+1} =S_m\,(\forall m\geq 1)\,, \quad \lambda = 0\,.
\end{equation}
Then the solution to (\ref{eqn:B1}) can be explicitly written down. In particular, for our specific $A$, the elements in $S$ and $S^{-1}$ have the form
\begin{equation}
S_{mn} = (-1)^{n-m}{{n-2}\choose{n-m}}\,,\quad S^{-1}_{mn} = {{n-2}\choose{n-m}}\quad\text{for}\quad n\geq m\geq 2\,.
\end{equation}
Then the solution takes the form:
\begin{equation}
\eta (t) = S \left(\begin{array}{cccccc}1  & \Ctt t & \frac{(\Ctt t)^2}{2}  &\cdots &\cdots & \frac{(\Ctt t)^l}{l!}\\ & 1 & \Ctt t &\cdots &\cdots & \frac{(\Ctt t)^{(l-1)}}{(l-1)!}\\& &\ddots &\ddots  & & \vdots \\&  &  & \ddots & & \Ctt t\\& &  & & & 1\end{array}\right)  S^{-1} 
\left(  \begin{array}{c} \eta_l(0) \\ \eta_{l-1}(0) \\ \vdots \\ \eta_1(0) \\ \eta_0(0) \end{array} \right) \,.
\end{equation}
Given that $\norm{\eta_l(0)} \leq \frac{H^l}{l!}$, we have
\begin{align} \label{eqn:11120}
\norm{\eta_l} & \leq   \frac{H^l}{l!} + \sum_{k=1}^l \frac{ (\Ctt t)^k}{k!(k-1)! }  \sum_{j=k+1}^{l+1} \frac{(j-2)!}{(j-k-1)!}  \frac{H^{l+1-j}}{(l+1-j)!}  \nonumber
\\ & =  \frac{H^l}{l!} + \sum_{k=1}^l \frac{ (\Ctt t)^k}{k!(k-1)! }   \frac{1}{(l+1)!}\sum_{j=k+1}^{l+1} \frac{(j-2)! j!}{(j-k-1)!}  \frac{(l+1)! H^{l+1-j}}{(l+1-j)! j!} 
\end{align}
Note that the inequality above is sharp. Considering $\frac{(j-2)!j!}{(j-k-1)!}$ is an increasing function in $j$ for $j\geq k+1 \geq 2$, we could further rewrite~\eqref{eqn:11120} estimates as
\begin{align}
\norm{\eta_l} &\leq  \frac{H^l}{l!} + \sum_{k=1}^l \frac{(\tilde{C_2}t)^k}{k! (k-1)!} \frac{(l-1)!}{(l-k)!} \sum_{j=k+1}^{l+1} \frac{(l+1)!}{j! (l+1-j)!} H^{l+1-j} 
\\ & \leq \frac{H^l}{l!} + (1+H)^{l+1} \sum_{k=1}^l \frac{(\tilde{C_2}t)^k}{k! (k-1)!} \frac{(l-1)!}{(l-k)!} \,. \label{eqn:322}
\end{align}
We end the proof by noting that the summation term on the right of \eqref{eqn:322} can be either estimated as
\begin{align*}
\sum_{k=1}^l \frac{(\tilde{C_2}t)^k}{k! (k-1)!} \frac{(l-1)!}{(l-k)!}  & \leq \sum_{k=1}^l (\tilde{C_2}t)^k \frac{(l-1)!}{(l-k)! k!} \leq (1+\tilde{C}_2 t)^l\,,
\end{align*}
or
\begin{align*}
\sum_{k=1}^l \frac{(\tilde{C_2}t)^k}{k! (k-1)!} \frac{(l-1)!}{(l-k)!}  & \leq e^{\tilde{C}_2 t} \sum_{k=1}^l \frac{(l-1)!}{(k-1)! (l-k)!} = e^{\tilde{C}t} 2^{l-1}\,.
 \end{align*}
\end{appendix}

\bibliographystyle{siam}
\bibliography{stochastic_AP_ref}

\end{document}